\documentclass{amsart}

\usepackage{amsmath}
\usepackage{amssymb}
\usepackage{amsfonts}
\usepackage{enumerate}
\usepackage{vmargin}

\newcommand{\bB}{{\mathbb{B}}}

\newcommand{\bN}{{\mathbb{N}}}

\newcommand{\bR}{{\mathbb{R}}}

  \newcommand{\C}{{\mathcal{C}}}
  
  \newcommand{\E}{{\mathcal{E}}}

  \newcommand{\M}{{\mathcal{M}}}
  \newcommand{\N}{{\mathcal{N}}}

\renewcommand{\phi}{\varphi}
\newcommand{\upchi}{{\raise.35ex\hbox{\ensuremath{\chi}}}}

\renewcommand{\leq}{\leqslant}
\renewcommand{\geq}{\geqslant}

\renewcommand{\ge}{\geqslant}

\newtheorem{thm}{Theorem}[section]

\newtheorem{cor}[thm]{Corollary}
\newtheorem{lemma}[thm]{Lemma}
\newtheorem{rk}[thm]{Remark}

\begin{document}

\title{An inequality in noncommutative $L_p$-spaces}
\date{}
\author[\'E. Ricard]{\'Eric Ricard}
\address{Laboratoire de Math{\'e}matiques Nicolas Oresme,
Universit{\'e} de Caen Normandie,
14032 Caen Cedex, France}
\email{eric.ricard@unicaen.fr}


\begin{abstract}
We prove that  for any (trace-preserving) 
conditional expectation $\E$ on a noncommutative $L_p$ with $p>2$, $Id-\E$  
is a contraction on the positive cone $L_p^+$.
\end{abstract}
\maketitle

\section{Introduction}

 It is plain that for positive real numbers $a, b\ge 0$ and
 $p\geq 2$, one has
$$ (a-b)(a^{p-1}-b^{p-1})\geq |a-b|^{p}.$$
Integrating the above inequality on some measure space $(\Omega,\mu)$ implies
that for $f,g\in L_p(\Omega,\mu)^+$, 
$$\int_\Omega \big(f(x)-g(x)\big)\big(f^{p-1}(x)-g^{p-1}(x)\big) {\rm
  d}\mu(x)\geq \big\| f-g\big\|_p^p.$$ In \cite{Must}, Mustapha
Mokhtar-Kharroubi notices that this inequality may be used to get
contractivity results on the positive cone of $L_p(\Omega,\mu)$. This
note originates from the question whether its noncommutative analogue
remains true. We provide a proof in the next section.
 We hope that the techniques involved there may be useful for 
further studies.
We end up by making explicit some results from \cite{Must}
for noncommutative $L_p$-spaces.

 We refer the reader to \cite{PX} for the definitions of $L_p$-spaces
 associated to semifinite von Neumann algebras or more general ones.
We also freely use basic results from \cite{Bha}.

\section{Results}
Let $(\M,\tau)$ be a semifinite von Neumann algebra.  We denote by
$f_p:\M^+\to\M^+$, the $p^{\rm th}$-power map and by $\M^{++}$ the set
of positive invertible elements. We will often refer to positivity of
the trace for the fact that if $a,b\in \M^+\cap L_1(\M)$, then 
$\tau(ab)\geq0$.

\begin{thm}\label{th}
Let $p\geq 2$ and $a,\,b\in L_p(\M)^+$, then 
$$\tau \big(|a-b|^p\big) \leq \tau\big((a-b)(a^{p-1}-b^{p-1})\big).$$
\end{thm}

\begin{proof}
First, as for any sequence of (finite) projections $p_i$ going to 1 strongly
in $\M$ and any $x\in L_q(\M)$ ($1\leq q<\infty$) $\|p_ixp_i-x\|_q\to 0$, we may assume that $\M$ is finite. Next by replacing $a$ and $b$ by $a+\varepsilon 1$ and $a+\varepsilon 1$ for some $\varepsilon>0$, we may also assume that $[a,b]\subset \M^{++}$ to avoid any unnecessary technical complication. 

We write $a=b+\delta$.
To prove the result, we distinguish according to the values of $p$.

\medskip

\noindent {\it Case 1: $p\in[2,3]$} 

\smallskip

{\it Case 1.a: $a\geq b$, i.e. $\delta\geq 0$.}

As $p-1=1+\theta$ with $\theta\in[0,1]$, we use the well known integral formula
\begin{equation}\label{intfor}
s^{1+\theta} = c_\theta \int_{\bR_+} \frac{t^\theta s^2}{s+t} \,
\frac{{\rm d}t}{t},\qquad \frac{s^2}{s+t}= s-t +\frac{t^2}{s+t} \,
.\end{equation}
Hence
$$\tau \big(\delta (a^{1+\theta}-b^{1+\theta})\big)=c_\theta \int_{\bR_+} t^\theta
\tau\Big( \delta \big( \delta + t^2 (b+\delta+t)^{-1}-t^2(b+t)^{-1}\big)\Big)\frac{{\rm d}t}{t}.$$
Recall the identity $(b+\delta+t)^{-1}-(b+t)^{-1}=-
(b+\delta+t)^{-1}\delta(b+t)^{-1}$. Using positivity of the trace
with $\delta(b+\delta+t)^{-1}\delta\leq \delta(\delta+t)^{-1}\delta$ and $(b+t)^{-1}\leq t^{-1}$:
$$\tau\Big( \delta \big( \delta + t^2 (b+\delta+t)^{-1}-t^2(b+t)^{-1}\big)\Big)
\geq \tau\big( \delta^2-t\delta(\delta+t)^{-1}\delta\big)=
\tau \big(\delta^3(\delta+t)^{-1}\big).$$
Integrating, we get the desired inequality  
$\tau \big(\delta (a^{1+\theta}-b^{1+\theta})\big)\geq \tau \big( \delta^{2+\theta}\big).$

\smallskip

{\it Case 1.b:} $\delta$ arbitrary with decomposition $\delta_+-\delta_-$ into positive and negative parts. We reduce it to the previous case by introducing $\alpha=a+\delta_-=b+\delta_+$, so that $\alpha\geq a,b$.
We have
\begin{eqnarray*}
\tau\big( (a-b)(a^{p-1}-b^{p-1})\big)&=&\tau\big( (a-\alpha)(a^{p-1}-\alpha^{p-1})\big)+\tau\big( (a-\alpha)(\alpha^{p-1}-b^{p-1})\big)\\ & & \;\;+\tau\big( (\alpha-b)(\alpha^{p-1}-b^{p-1})\big)+\tau\big( (\alpha-b)(a^{p-1}-\alpha^{p-1})\big).
\end{eqnarray*}
The first and the third terms are bigger than $\tau\big( \delta_-^p\big)$ and 
$\tau\big( \delta_+^p\big)$ by Case 1.a. Hence it suffices to check that the two remaining terms are positive. We use again the integral formula \eqref{intfor} and $\delta_+\delta_-=0$ and positivity of the trace
\begin{eqnarray*}
\tau\big( -\delta_-(\alpha^{p-1}-b^{p-1})\big)&=&-c_\theta \int_{\bR_+} t^\theta
\tau\Big( \delta_-\big( \delta_+ + t^2 (b+\delta_++t)^{-1}-t^2(b+t)^{-1}\big)\Big)\frac{{\rm d}t}{t}\\
&=& c_\theta \int_{\bR_+} t^\theta t^2
\tau\Big( \delta_-\big(  (b+t)^{-1}-(b+\delta_++t)^{-1}\big)\Big)\frac{{\rm d}t}{t}\geq 0.
\end{eqnarray*}
The last term is handled similarly.

\medskip

\noindent {\it Case 2: $p\geq 3$.}
 First, for any $n\in \bN$, $n\geq 1$, one easily checks by induction that we have the following identity
$$\tau\big((a-b)(a^{p-1}-b^{p-1})\big)=
\tau\big(\delta\big( (b+\delta)^{p-1-n}-b^{p-1-n}\big)b^{n}\big)+
\sum_{k=1}^n  \tau \big(\delta(b+\delta)^{p-1-k}\delta b^{k-1}\big).$$
Let $n\geq1$ be so that $p-1-n=1+\theta$ with $\theta\in[0,1[$. By positivity 
of the trace, we get
$$\tau\big((a-b)(a^{p-1}-b^{p-1})\big)\geq
\tau\big(\delta\big( (b+\delta)^{1+\theta}-b^{1+\theta}\big)b^{n}\big)+
  \tau \big(\delta^2(b+\delta)^{p-2}).$$
By the same computations as above thanks to \eqref{intfor}
\begin{eqnarray*}\tau\big(\delta\big( (b+\delta)^{1+\theta}-b^{1+\theta}\big)b^{n}\big)&=&
c_\theta \int_{\bR_+} t^\theta \tau\Big( \delta \big( \delta + t^2
(b+\delta+t)^{-1}-t^2(b+t)^{-1}\big)b^n\Big)\frac{{\rm d}t}{t}\\ &=&
c_\theta \int_{\bR_+} t^\theta\tau\Big( \delta \big( \delta - t^2
(b+\delta+t)^{-1}\delta(b+t)^{-1}\big)b^n\Big)\frac{{\rm
    d}t}{t}\\ &\geq & c_\theta \int_{\bR_+} t^\theta\tau\Big( \delta
\big( \delta - t\delta(b+t)^{-1}\big)b^n\Big)\frac{{\rm
    d}t}{t}\\ &=&c_\theta \int_{\bR_+} t^\theta\tau\Big( \delta^2
b(b+t)^{-1}b^n\Big)\frac{{\rm d}t}{t}\\&=& \tau\big( \delta^2
b^{n+\theta}\big)=\tau\big( \delta^2
b^{p-2}\big),
\end{eqnarray*}
where we used again positivity of the trace with $(b+\delta+t)^{-1}\leq t^{-1}$
and $0\leq (b+t)^{-1}b^n$.

To conclude let $\E$ to be the conditional expectation onto the
subalgebra $\N=\{\delta\}''$. As $\N$ is commutative,  the
Jensen inequality is valid; for any $\alpha\geq 1$ and $x\in \M^+$:
$\E (x^\alpha)\geq (\E x)^\alpha$. With $\alpha=p-2\geq1$,
$$\tau\big( \delta^2 b^{p-2}\big)=\tau\big( \delta^2
\E(b^{p-2})\big)\geq \tau\big( \delta^2 (\E b)^{p-2}\big), \qquad
\tau\big( \delta^2 (b+\delta)^{p-2}\big)\geq \tau\big( \delta^2
(\E(b+\delta))^{p-2}\big).$$ But with the usual decomposition
$\delta=\delta_+-\delta_-$, as $a,\,b\geq 0$, $\E b \geq \delta_-$ and
$\E (b+\delta) \geq \delta_+$. By commutativity of $\N$, we can
conclude 
$$\tau\big((a-b)(a^{p-1}-b^{p-1})\big)\geq \tau \big( \delta^2(\delta_-^{p-2}+\delta_+^{p-2})\big)=\tau \big( |\delta|^p\big).$$
\end{proof}

We provide an alternative proof when $p\in[3,4]$.

 Denote by $R_x$ and $L_x$ the
right and left multiplication operators by $x\in \M$ defined on all
$L_p(\M)$ ($1\leq p\leq \infty$). When $p=2$, for any $x\in \M^{sa}$,
the $C^*$-algebra generated in $\bB(L_2(\M))$ by $L_x$ and $R_x$ is
commutative and isomorphic to $\C(\sigma(x)\times \sigma(x))$ where $\sigma(x)$ is the spectrum of $x$.

\begin{lemma}\label{diff} For $p\geq1$, the map $f_p$ is  Fr\'echet differentiable on $\M^{++}$. For $\M$ finite, the derivative is given by the formula in $L_2(\M)$: 
$$\forall x\in \M^{++}, \forall h\in \M^{sa},\qquad 
{\rm D}_{x}f_p(h)=p\int_0^1 \Big( t L_x +(1-t)R_x\Big)^{p-1}(h) \;{\rm d}t.$$ 
\end{lemma}

\begin{proof} 
Assume $x\geq 3 \delta$ for some $\delta>0$. Taking $h\in \M^{sa}$ with 
$\|h\| < \delta$, we may compute 
$f_p(x+h)$ using the holomorphic functional calculus by choosing a curve $\gamma$ with index 1 surrounding the spectrum of $x$ with 
$\gamma\subset \{ z \;|\; {\rm Re} z>0\}$ and ${\rm dist}(\gamma,\sigma(x))\geq2\delta$:
$$ (x+h)^p=\frac 1{2i\pi} \int_\gamma \frac {z^p}{z- (x+h) } {\rm d}z$$
Hence 
$$(x+h)^p-x^p= \frac 1{2i\pi} \int_\gamma z^p 
\big(z- (x+h)\big)^{-1}h\big(z- x\big)^{-1}{\rm d}z$$
It follows directly that $f_p$ is  Fr\'echet differentiable with derivative
$$ {\rm D}_{x}f_p(h)=\frac 1{2i\pi} \int_\gamma z^p \big(z-
x\big)^{-1}h\big(z- x\big)^{-1}{\rm d}z=\frac 1{2i\pi} \int_\gamma
z^p L_{(z- x)^{-1}}R_{(z- x)^{-1}}(h)\;{\rm d}z.$$ It then suffices to check
that the two formulas coincide when $\M$ is finite; as $\M\subset
L_2(\M)$, we do it for $h\in L_2(\M)$. But in $\bB(L_2(\M))$, this boils down to an equality in $\C(\sigma(x)\times \sigma(x))$ so that
we need only to justify that
$$\forall a,b\in \bR^{+*}, \qquad \frac 1{2i\pi} \int_\gamma \frac {z^p}{(z-a)(z-b) } {\rm d}z=p\int_0^1 \big( t a +(1-t) b\big)^{p-1} \;{\rm d}t.$$
The above computations yield that the left-hand side  is $\frac{a^p-b^p}{a-b}$ if $a\neq b$ and $pa^{p-1}$ if $a=b$ which clearly coincide with the right-hand side. 
\end{proof}

Assuming $\M$ finite and $a=b+\delta,b\in \M^{++}$ as above, the alternative proof when $p\in[3,4]$ relies on Lemma \ref{diff}:
\begin{eqnarray*}
\tau\big( (a-b)(a^{p-1}-b^{p-1})\big)&=&p\int_0^1 
\int_0^1 \tau\Big( \delta \Big( t L_{b+u\delta} +(1-t)R_{b+u\delta}\Big)^{p-2}(\delta) \Big)\;{\rm d}t\;{\rm d}u\\
&=&p\int_0^1 \int_0^1 \langle \delta, \Big( t L_{b+u\delta} +(1-t)R_{b+u\delta}\Big)^{p-2} (\delta)\rangle_{L_2(\M)}\;{\rm d}t\;{\rm d}u
\end{eqnarray*}

As $p-2\in [1,2]$, $f_{p-2}$ is operator convex, so that for any 
$m\in \bB(L_2(\M))^+$ and any projection $\E\in  \bB(L_2(\M))$, we have
$\E m^{p-2}\E \geq \big(\E m \E)^{p-2}$. We choose $\E$ to be the 
$L_2$-conditional expectation onto the subalgebra generated by $\delta$.
\begin{eqnarray*}
\langle \delta, \Big( t L_{b+u\delta} +(1-t)R_{b+u\delta}\Big)^{p-2} (\delta)\rangle_{L_2(\M)} &=& \langle \delta, \E \Big( t L_{b+u\delta} +(1-t)R_{b+u\delta}\Big)^{p-2}\E (\delta)\rangle_{L_2(\M)}\\
&\geq & \langle \delta, \Big( t \E L_{b+u\delta}\E  +(1-t)\E R_{b+u\delta}\E \Big)^{p-2}(\delta)\rangle_{L_2(\M)}\nonumber\\
&=& \langle \delta, \Big( t  L_{\E(b)+u\delta}\E  +(1-t) R_{\E(b)+u\delta}\E \Big)^{p-2}(\delta)\rangle_{L_2(\M)}\nonumber\\
&=& \langle \delta, \Big( t  L_{\E(b)+u\delta}  +(1-t) R_{\E(b)+u\delta} \Big)^{p-2}(\delta)\rangle_{L_2(\M)},
\end{eqnarray*}
where in the last equality we have used that $R_x$ and $\E$ commute if $x\in \delta''$. Tracking back the equalities, we obtain
$$\tau\big( \delta((b+\delta)^{p-1}-b^{p-1})\big)\geq 
\tau\big( \delta((\E(b)+\delta)^{p-1}-\E(b)^{p-1})\big)\geq \tau\big( |\delta|^p\big),$$
where the last inequality comes from the result in the commutative case.

\begin{rk} {\rm We point out that, for $p\in ]2,3[$, the result cannot be reduced to the commutative case as in the alternative proof. Indeed, $t\mapsto t^{p-2}$ is operator concave and  the first inequality right above  reverses.}
\end{rk}
\begin{rk} {\rm 
Let $\phi: L_p(\M)\to L_{p'}(\M)$ be the duality map so that 
 $\langle x,\phi(x)\rangle_{L_p(\M),L_{p'}(\M)}=\|x\|_p^p$ and 
$\|\phi(x)\|_{p'}= \|x\|_p^{p-1}$. When restricted to $L_p(\M)^+$, it is exactly $f_{p-1}$, so the result can be written as: for $a,\,b\in L_p(\M)^+$
$$\langle a-b,\phi(a)-\phi(b)\rangle_{L_p(\M),L_{p'}(\M) } \geq \big\|
a-b\big\|_p^p.$$
In this form, the inequality extends to general $L_p$-spaces in the sense of Haagerup, 
see \cite{Rmaz, HJX} for the arguments.}
\end{rk}

\begin{cor}\label{corexp} Let $(\M,\tau)$ be a semifinite von Neumann algebra and 
$\E:\M\to \M$ be a $\tau$-preserving conditional expectation, then for
  all $ p\geq2$ and $x\in L_p(\M)^+$,
\begin{equation}\label{eq1}\big\|x-\E x\big\|_p\leq   \big\|x\big\|_p.
\end{equation}
\end{cor}

\begin{proof}
Apply the above theorem with  $a=x$ and $b=\E x$, as 
$\tau\big( (x-\E x)(\E x)^{p-1}\big)=0$, the H\"older inequality gives:
$$\big\|x-\E x\big\|_p^p\leq   \tau\big( (x-\E x)  x^{p-1}\big)\leq \big\|x-\E x\big\|_p \big\|x\big\|_{p}^{p-1}.$$
\end{proof}

\begin{rk}{\rm  The inequality \eqref{eq1} does not hold for $p<2$;
 a counterexample with $\M=\ell_\infty^2$ can be found in \cite{Must}.
There,  a slight extension of \eqref{eq1} is given; one can replace
$\E$ by any positive contractive projection $\C$ on $L_p(\M)$.
}
\end{rk}

As explained in \cite{Must}, the main inequality applies more
generally to semigroups. 

\begin{cor}
Let $(\M,\tau)$ be a semifinite von Neumann algebra and $(T_t)_{t\geq
  0}$ be a  trace preserving unital positive strongly continuous  semigroup on
$\M$. For $\lambda>0$, let $R_{\lambda}=
\int_0^\infty e^{-\lambda t} T_t {\rm d}t$ be its resolvent.
Then 
for all $ p\geq 2$, $\lambda>0$ and $x\in L_p(\M)^+$, 
$$\big\|x-\lambda R_\lambda x\big\|_p\leq   \big\|x\big\|_p.$$ 
\end{cor}
\begin{proof}
We proceed as in Corollary \ref{corexp}. We apply Theorem \ref{th} with
$a=x$ and $b=\lambda R_\lambda x$ to get
\begin{eqnarray*}
\big\|x- \lambda R_\lambda x\big\|_p^p &\leq&   \tau\big( (x-\lambda R_\lambda x)  
x^{p-1}\big)-\tau\big( (x-\lambda R_\lambda x)  
(\lambda R_\lambda x)^{p-1}\big)\\ &\leq&
 \big\|x-\lambda R_\lambda x\big\|_p \big\|x\big\|_{p}^{p-1}-\tau\big( (x-\lambda R_\lambda x)  
(\lambda R_\lambda x)^{p-1}\big).
\end{eqnarray*}
To conclude, it suffices to note  that $\tau\big( (x-\lambda R_\lambda x)(R_\lambda x)^{p-1}\big)\geq 0$. 

It can be checked by approximations thanks to the resolvent formula: $x-\lambda R_\lambda x=\lim_{t\to \infty
}t(1-tR_t)R_\lambda x$. Indeed, recall that $tR_t$ is positive unital and trace preserving and hence 
a contraction on $L_p$ so that 
$$\tau\big( tR_t(R_\lambda x). (R_\lambda
x)^{p-1}\big)\leq \big\| R_\lambda x\big\|_p^p\qquad \textrm{and}\qquad
\tau\big( \big(t(1-tR_t)R_\lambda x\big)(R_\lambda x)^{p-1}\big)\geq 0.$$
\end{proof}
\bibliographystyle{plain}

\begin{thebibliography}{1}

\bibitem{Bha}
Rajendra Bhatia.
\newblock {\em Matrix analysis}, volume 169 of {\em Graduate Texts in
  Mathematics}.
\newblock Springer-Verlag, New York, 1997.

\bibitem{HJX}
Uffe Haagerup, Marius Junge, and Quanhua Xu.
\newblock A reduction method for noncommutative {$L_p$}-spaces and
  applications.
\newblock {\em Trans. Amer. Math. Soc.}, 362(4):2125--2165, 2010.

\bibitem{Must}
Mustapha Mokhtar-Kharroubi.
\newblock Contractivity theorems in real ordered Banach spaces with
  applications to relative operator bounds, ergodic projections and conditional
  expectations.
\newblock Preprint, https://hal.archives-ouvertes.fr/hal-01148968.

\bibitem{PX}
Gilles Pisier and Quanhua Xu.
\newblock Non-commutative {$L^p$}-spaces.
\newblock In {\em Handbook of the Geometry of {B}anach Spaces, {V}ol.\ 2},
  pages 1459--1517. North-Holland, Amsterdam, 2003.

\bibitem{Rmaz}
\'Eric Ricard.
\newblock H\"older estimates for the noncommutative Mazur maps.
\newblock {\em Arch. Math. (Basel)}   104  (2015),  no. 1, 37--45.

\end{thebibliography}

\end{document}